\documentclass[11pt, a4paper]{amsart}

\usepackage[T1]{fontenc}
\usepackage{lmodern}
\usepackage{amssymb, amsmath, amsthm, bbm}
\usepackage{hyperref} 

\theoremstyle{plain}
\newtheorem{thm}[subsubsection]{Theorem}
\newtheorem{prop}[subsubsection]{Proposition}
\newtheorem{lem}[subsubsection]{Lemma}
\newtheorem{cor}[subsubsection]{Corollary}
\newtheorem{defin}[subsubsection]{Definition}

\theoremstyle{remark}
\newtheorem{rem}[subsubsection]{Remark}
\newtheorem{exam}[subsubsection]{Example}

\newcommand{\ve}{\varepsilon}
\newcommand{\ud}{\, \mathrm{d}}

\DeclareMathOperator{\dist}{dist}
\DeclareMathOperator{\supp}{supp}
\DeclareMathOperator{\card}{card}

\title[]{Functions with ultradifferentiable powers}

\author{Vincent Thilliez}

\address{Laboratoire Paul Painlev\'e\\
Universit\'e de Lille\\ 
Math\'ematiques - B\^atiment M2\\
F-59655 Villeneuve d'Ascq Cedex, France}

\email{vincent.thilliez@univ-lille.fr}


\begin{document}

\begin{abstract} 
We study the regularity of smooth functions $f$ defined on an open set of $\mathbb{R}^n$ and such that, for certain integers $p\geq 2$, the powers $f^p :x\mapsto (f(x))^p$ belong to a Denjoy-Carleman class $\mathcal{C}_M$ associated with a suitable weight sequence $M$. Our main result is a statement analogous to a classic theorem of H. Joris on $\mathcal{C}^\infty$ functions: if a function $f:\mathbb{R}\to\mathbb{R}$ is such that both functions $f^p$ and $f^q$ with $\gcd(p,q)=1$ are of class $\mathcal{C}_M$ on $\mathbb{R}$, and if the weight sequence $M$ satisfies the so-called moderate growth assumption, then $f$ itself is of class $\mathcal{C}_M$. Various ancillary results, corollaries and examples are presented. 
\end{abstract}

\maketitle

\section*{Introduction}  
It is generally difficult to relate the regularity of a real or complex-valued function $f$ defined on an open set of $\mathbb{R}^n$ to regularity assumptions on some of its powers $f^p :x\mapsto (f(x))^p$ with $p\in \mathbb{N}$, $p\geq 2$. However, in 1982, H. Joris \cite{Jor} proved the following striking result: if a function $f:\mathbb{R}\to \mathbb{R}$ is such that both functions $f^2$ and $f^3$, or more generally $f^p$ and $f^q$ with $\gcd(p,q)=1$, are of class $\mathcal{C}^\infty$ on $\mathbb{R}$, then $f$ itself is of class $\mathcal{C}^\infty$. As pointed out in \cite{DKP, JP}, the result also holds for complex-valued functions. Various generalizations were subsequently established around the notion of pseudo-immersion \cite{DKP, JP, Rai}. 

In spite of its innocent-looking statement, Joris's theorem is not easy to establish. The original proof involved an intricate study of the vanishing of the derivatives of $f$ at points of flatness, based on  combinatorial relations arising from the Fa\`a di Bruno formula. 

However, a much simpler and shorter proof was published in 1989 by I. Amemyia and K. Masuda \cite{AM}. Its key argument is an algebraic lemma stating that the ring of power series with coefficient in a ring $R$ inherits a suitable property of $R$ relative to powers of its elements. 

Unexpectedly, in 2018, as Joris's theorem was discussed on the \emph{MathOverflow} website, the anonymous contributor nicknamed  ``fedja'' 
outlined a remarkable alternative proof based on a characterization of smooth functions on the real line by holomorphic approximation. Fedja's argument \cite{Fed} actually yields an even stronger result, as it works for finite differentiability classes: roughly speaking, given $p$ and $q$ with $\gcd(p,q)=1$, there is an integer $m$, depending only on $p$ and $q$, such that for $k$ large enough, the function $f$ is of class $\mathcal{C}^k$ as soon as $f^p$ and $f^q$ are of class $\mathcal{C}^{mk}$, and the proof provides crude estimates for $m$.\\

The main goal of the present paper is to show that the property described by Joris's theorem holds in Denjoy-Carleman ultradifferentiable classes $\mathcal{C}_M$, provided the weight sequence $M$ that defines the class satisfies the so-called \emph{moderate growth} assumption. Our approach will follow closely the path of the aforementioned proof of Fedja \cite{Fed}, while making suitable modifications needed in the Denjoy-Carleman setting.\\

The paper is organized as follows. 

Section \ref{DCclasses} gathers the definitions and required material pertaining to weight sequences and Denjoy-Carleman classes. 

Section \ref{exposit} begins with a review of some known results on the regularity of $\mathcal{C}^\infty$ functions $f:\mathbb{R}\to\mathbb{R}$ such that $f^p$ is of class $\mathcal{C}_M$ for a given integer $p\geq 2$. Incidentally, Proposition \ref{answer} answers a question asked in \cite{Th3}. These mostly negative results serve as a motivation for a $\mathcal{C}_M$ version of Joris's theorem, which is  stated in the second part of Section \ref{exposit} (Theorem \ref{main}). Various comments and corollaries are then given. In particular, the case of functions of several variables is briefly discussed. 

Sections \ref{technical} and \ref{final} are entirely devoted to the proof of Theorem \ref{main}. In Section \ref{technical}, we gather the main technical ingredients needed in the proof. In particular, an approximation-theoretic characterization of $\mathcal{C}_M$ regularity on a real interval is established; this result (Proposition \ref{approx}) may be of independent interest. In Section \ref{final}, the technical tools of Section \ref{technical} are finally used to complete the proof of Theorem \ref{main}, following the general pattern of Fedja's argument \cite{Fed}. 
 
\section{Denjoy-Carleman classes}\label{DCclasses}
\subsection{Some properties of sequences}\label{sequences}
A sequence $M=(M_j)_{j\geq 0}$ of positive real numbers will be called a \emph{weight sequence} if it satisfies the following assumptions:
\begin{equation}\label{norm}
M \text{ is increasing and } M_0=1,
\end{equation}
\begin{equation}\label{logc}
M \text{ is logarithmically convex},
\end{equation}
\begin{equation}\label{nonana}
\lim_{j\to\infty} (M_j)^{1/j}=\infty.
\end{equation}
Property \eqref{logc} amounts to saying that the sequence $(M_{j+1}/M_j)_{j\geq 0}$ is nondecreasing. Together with \eqref{norm}, it implies
\begin{equation*}
M_jM_k\leq M_{j+k}\ \textrm{ for any } (j,k)\in\mathbb{N}^2.
\end{equation*}
We say that a weight sequence $M$ has \emph{moderate growth} if there is a positive constant $A$ such that we have
\begin{equation}\label{modg}
M_{j+k}\leq A^{j+k} M_jM_k\ \textrm{ for any } (j,k)\in\mathbb{N}^2. 
\end{equation}
We say that a weight sequence $M$ satisfies the \emph{strong non-quasianalyticity} condition if there is a positive constant $A$ such that we have
\begin{equation}\label{snqa}
\sum_{j\geq k}\frac{M_j}{(j+1)M_{j+1}}\leq A \frac{M_k}{M_{k+1}} \textrm{ for any } k\in\mathbb{N}.
\end{equation}
Property \eqref{snqa} obviously implies the classical Denjoy-Carleman \emph{non-quasiana\-lyt\-icity} condition
\begin{equation}\label{nqa}
\sum_{j\geq 0}\frac{M_j}{(j+1)M_{j+1}}<\infty.
\end{equation}
A weight sequence $M$ is said to be \emph{strongly regular} if it satisfies \eqref{modg} and \eqref{snqa}.  

\begin{exam}\label{exgev}
Let $ \alpha $ and $\beta$ be real numbers, with $\alpha> 0$. One can define a strongly regular weight sequence $M$ by setting $M_j=(j!)^\alpha(\ln j)^{\beta j}$ for $j$ large enough and choosing suitable first terms. This is the case, in particular, for Gevrey sequences $M_j=(j!)^\alpha$. 
\end{exam}

\begin{exam}
For any real $\beta>0$, one can also define a weight sequence $M$ with $M_j=(\ln j)^{\beta j}$ for $j$ large enough. This sequence has moderate growth, and it satisfies the non-quasianalyticity property \eqref{nqa} if and only if $\beta>1$. It does not satisfy the strong non-quasianalyticity property \eqref{snqa}.
\end{exam}

\begin{exam}\label{qgev}
For any real $\lambda>0$, the weight sequence $M^\lambda$ defined by $M^\lambda_j=\exp\big(\frac{\lambda}{4}j^2\big)$ satisfies \eqref{snqa} but it does not have moderate growth. The sequences $M^\lambda$ will reappear in the examples of Section \ref{exposit}. 
\end{exam}

With every weight sequence $M$, it is a standard procedure to associate the function $h_M$ defined by $h_M(t)=\inf_{j\geq 0}t^jM_j $ for any real $ t>0 $, and $ h_M(0)=0 $. Using \eqref{norm}, \eqref{logc} and \eqref{nonana}, it is easy to see that  
$h_M(t)=t^jM_j$ for $j\geq 1$ and $\frac{M_j}{M_{j+1}}\leq t< \frac{M_{j-1}}{M_j}$, and $ h_M(t)=1 $ for $t\geq 1/M_1 $. In particular, $h_M$ is continuous, nondecreasing and it fully determines $M$ since we have 
\begin{equation*}
M_j=\sup_{t>0}t^{-j}h_M(t)\, \text{ for any }\, j\in\mathbb{N}.
\end{equation*}
Setting $t_j=\frac{M_j}{M_{j+1}}$, we also obtain 
\begin{equation}\label{Legendre2}
M_j= t_j^{-j}h_M(t_j)\, \text{ with }\, \lim_{j\to\infty}t_j=0.
\end{equation}

\begin{exam} 
Let $M$ be as in Example \ref{exgev}, and set $\eta(t)=\exp(-(t\vert\ln t\vert^\beta)^{-1/\alpha})$ for $t>0$ small enough. Elementary computations show that there are constants $a>0$, $b>0$ such that $\eta(at)\leq h_M(t)\leq \eta(bt)$ as $t$ tends to $0$. 
\end{exam} 
 
It can be derived from \cite[Proposition 3.6]{Kom} that the moderate growth assumption \eqref{modg} is equivalent to the existence, for any real $s\geq 1$, of a constant $\kappa_s\geq 1 $ such that
\begin{equation}\label{hfunct2}
h_M(t)\leq \big(h_M(\kappa_s t)\big)^s\text{ for any }t\geq 0.
\end{equation}
Other equivalent conditions for \eqref{modg}, or for the strong non-quasianalyticity property \eqref{snqa}, can be found in the state-of-the-art study of weight sequences and weight functions carried out in the recent works \cite{Jim, JSS1, JSS2}, originating in J. Sanz's work on proximate orders \cite{San}. 

As a consequence of \eqref{hfunct2} and of the definition of $h_M$, it is easy to see that if a weight sequence $M$ has moderate growth, then we have 
\begin{equation}\label{hfunct3}
t^{-j}h_M(t)\leq \kappa_2^jM_j h_M(\kappa_2t)\text{ for any }t> 0 \text{ and any }j\in \mathbb{N}.
\end{equation}

\subsection{Definition of Denjoy-Carleman classes} 
In what follows, we denote the length $j_1+\cdots+j_n$ of a multi-index $J=(j_1,\ldots,j_n)\in\mathbb{N}^n$ by the corresponding lower case letter $j$, and we put $\partial^J=\partial^j/\partial x_1^{j_1}\cdots\partial x_n^{j_n}$. 

Let $\Omega$ be an open subset of $\mathbb{R}^n$, and let $M$ be a weight sequence. We say that a $\mathcal{C}^\infty$ function $f:\Omega\to \mathbb{C}$ belongs to the \emph{Denjoy-Carleman class} $\mathcal{C}_M(\Omega)$ if for any compact subset $X$ of $\Omega$, one can find a real number $\sigma>0$ and a constant $C\geq 0$ such that
\begin{equation}
\vert \partial^Jf(x)\vert \leq C\sigma^j j!M_j\  \text{ for any }\, J\in\mathbb{N}^n\, \text{ and }\, x\in X.
\end{equation}
A germ of function at the origin in $\mathbb{R}^n$ is said to be of class $\mathcal{C}_M$ if it has a representative in $\mathcal{C}_M(\Omega)$ for some open neighborhood $\Omega$ of $0$. We denote by $\mathcal{C}_M(\mathbb{R}^n,0)$ the set of all such germs. 

Corresponding definitions for functions on segments of $\mathbb{R}$ instead of an open set will be needed. 
Given a segment $[a,b]$ of $\mathbb{R}$, a real number $\sigma>0$, and a $\mathcal{C}^\infty$ function $f:[a,b]\to \mathbb{C}$, we set
\begin{equation*}
\Vert f\Vert_{[a,b],\sigma}=\sup_{x\in [a,b],\ j\in \mathbb{N}}\frac{\vert f^{(j)}(x)\vert}{\sigma^j j!M_j}. 
\end{equation*}
We then say that the function $f$ belongs to the space $\mathcal{C}_{M,\sigma}([a,b])$ if it satisfies $\Vert f\Vert_{[a,b],\sigma}<\infty$. It is easy to see that $\mathcal{C}_{M,\sigma}([a,b])$ is a Banach space for the norm $\Vert\cdot\Vert_{[a,b],\sigma}$. Finally, we define the \emph{Denjoy-Carleman class} $\mathcal{C}_M([a,b])$ as the reunion of all spaces $\mathcal{C}_{M,\sigma}([a,b])$ for $\sigma>0$. Given an open subset $\Omega$ of $\mathbb{R}$, it is clear that a function $f:\Omega\to \mathbb{C}$ belongs to $\mathcal{C}_M(\Omega)$ if and only if its restriction to every segment $[a,b]$ contained in $\Omega$ belongs to $\mathcal{C}_M([a,b])$. 

We end this section with a brief review of the relationship between conditions on the sequence $M$ and properties of the corresponding classes; we refer to \cite{Th2} for details and references. Conditions \eqref{norm} and \eqref{logc} imply that $\mathcal{C}_M(\Omega)$, $\mathcal{C}_M(\mathbb{R}^n,0)$ and $\mathcal{C}_M([a,b])$ are algebras, and that $\mathcal{C}_M$ regularity is stable under composition. Condition \eqref{nonana} ensures that $\mathcal{C}_M(\Omega)$ (resp. $\mathcal{C}_M(\mathbb{R}^n,0)$)  strictly contains the algebra of real-analytic functions in $\Omega$ (resp. real-analytic germs at the origin). The moderate growth assumption \eqref{modg} can be interpreted in terms of stability of $\mathcal{C}_M$ regularity under the action of so-called ultradifferential operators; see \cite{Kom}. It clearly implies the weaker condition 
\begin{equation}\label{stabder}
M_{j+1}\leq A^{j+1}M_j \ \textrm{ for any } j\in\mathbb{N}
\end{equation}
which characterizes the stability of $\mathcal{C}_M$ classes under derivation. The non-quasi\-an\-a\-lyt\-icity property \eqref{nqa} characterizes the existence of a non-trivial element of $\mathcal{C}_M(\mathbb{R}^n,0)$ which is flat at $0$, whereas the stronger condition \eqref{snqa} is a necessary and sufficient condition for a $\mathcal{C}_M$ version of Borel's extension theorem. 

\section{Functions with ultradifferentiable powers}\label{exposit}
\subsection{Background and known results}\label{background} 
Let $M$ be a weight sequence and let $f$ be a germ of complex-valued 
function of class $\mathcal{C}^\infty$ at the origin in $\mathbb{R}$. Assume that there is an integer $p\geq 2$ such that the germ $f^p: x\mapsto (f(x))^p$ belongs to $\mathcal{C}_M(\mathbb{R},0)$. 
As observed in \cite[Remark 1]{Th3}, it is not difficult the check that if $\mathcal{C}_M(\mathbb{R},0)$ is stable under derivation and quasianalytic, then $f$ also belongs to $\mathcal{C}_M(\mathbb{R},0)$. 
This is no longer true in the non-quasianalytic case: indeed, for any real $\lambda>0$, set
\begin{equation}\label{glam}
g_\lambda(x)=\exp\left(-\frac{1}{\lambda}(\ln x)^2\right)\, \text{ for }\, x>0\, \text{ and }\, g_\lambda(x)=0\, \text{ for }\,x\leq 0.
\end{equation}
The proof of \cite[Lemma 1]{Th3} shows that $g_\lambda$ belongs to $\mathcal{C}_{M^\lambda}(\mathbb{R},0)$, where $M^\lambda$ is defined in Example \ref{qgev}, but not to any strictly smaller ring $\mathcal{C}_M(\mathbb{R},0)$. In particular, for $f=g_{p\lambda}$, we see that $f^p$ belongs to $\mathcal{C}_{M^\lambda}(\mathbb{R},0)$ whereas $f$ does not. Thus, the result fails for the weight sequences $M^\lambda$, even though the associated classes are stable under derivation and strongly non-quasianalytic. Since $M^\lambda$ does not have moderate growth, it was asked in \cite{Th3} whether the result would hold for tamer sequences $M$, namely strongly regular ones. The answer is still negative, as shown by the following proposition. 

\begin{prop}\label{answer}
Let $M$ be a strongly regular weight sequence. For every integer $p\geq 2$, there is a smooth function germ $f$ at the origin in $\mathbb{R}$ such that $f^p\in\mathcal{C}_M(\mathbb{R},0)$ and $f\notin \mathcal{C}_M(\mathbb{R},0)$.   
\end{prop}
\begin{proof} We start with a counter-example in two variables, slightly generalizing a construction of \cite{Th1}. By \cite[Lemma 3.6]{Th1b}, there is an element $\eta$ of $\mathcal{C}_M(\mathbb{R})$ which  vanishes at infinite order at the origin and satisfies $\eta(t)\geq h_M(b\vert t\vert)$ for some suitable constant $b>0$. Given an integer $m\geq 2$, we then set, for $(x,y)\in \mathbb{R}^2$, 
\begin{equation*}
F(x,y)=(x^2+y^{2m})\left(1+\frac{x^2\eta(y)}{x^2+y^{2m}}\right)^{1/p}.
\end{equation*}
Since $\eta$ is flat at $0$, the $\mathcal{C}^\infty$-smoothness of $F$ is immediate. Moreover, we have 
$(F(x,y))^p=(x^2+y^{2m})^p+ x^2(x^2+y^{2m})^{p-1}\eta(y)$, hence $F^p \in \mathcal{C}_M(\mathbb{R}^2,0)$. 
Using the power series expansion of $(1+t)^{1/p}$, we obtain, for $(x,y)$ close enough to $(0,0)$, the expansion 
\begin{equation*}
F(x,y)=x^2+y^{2m}+ \frac{1}{p}x^2\eta(y)+\sum_{j=1}^{+\infty}(-1)^ja_j\frac{x^{2j+2}}{y^{2mj}}\left(1+\frac{x^2}{y^{2m}}\right)^{-j}(\eta(y))^{j+1}
\end{equation*}
with $a_j= \frac{(p-1)(2p-1)\cdots (jp-1)}{p^{j+1}(j+1)!}$ for $j\geq 1$. Assume $0\leq x< y^m$. Expanding $\left(1+\frac{x^2}{y^{2m}}\right)^{-j}$ in power series, we then obtain the absolutely convergent expansion
\begin{equation}\label{expan1}
F(x,y)=G(x,y)+\sum_{j=1}^{+\infty}\sum_{k=0}^{+\infty}(-1)^{j+k}a_j \binom{j+k-1}{j-1}\frac{x^{2(j+k)+2}}{y^{2m(j+k)}}(\eta(y))^{j+1}
\end{equation}
with $G(x,y)= x^2\big(1+\frac{1}{p}\eta(y)\big)+y^{2m}$. We set $l=j+k$ and exchange the order of summation, so that \eqref{expan1} becomes 
\begin{equation}\label{expan}
F(x,y)=G(x,y)+\sum_{l=1}^{+\infty}(-1)^l c_l(y)x^{2l+2} \, \text{ for }\, 0\leq x<y^m, 
\end{equation}
with 
\begin{equation*}
c_l(y)=y^{-2ml}\sum_{j=1}^l a_j \binom{l-1}{j-1}(\eta(y))^{j+1}\, \text{ for }\, l\geq 1.
\end{equation*}
Clearly, \eqref{expan} implies 
\begin{equation*}
\frac{\partial^{2l+2} F}{\partial x^{2l+2}}(0,y)=(-1)^l(2l+2)!\, c_l(y)\, \text{ for }\, y>0\, \text{ and }\,  l\geq 1.
\end{equation*}
Observe that $c_l(y)\geq y^{-2ml}a_1 (\eta(y))^2\geq a_1(y^{-ml}h_M(by))^2$. Moreover, by \eqref{Legendre2}, there is a sequence $(y_l)_{l\geq 0}$ of positive real numbers such that $\lim_{l\to\infty} y_l=0$ and $h_M(by_l)=(by_l)^{ml}M_{ml}$, hence $c_l(y_l)\geq a_1 b^{2ml}(M_{ml})^2$. Using \eqref{logc} and \eqref{modg}, we also have $(M_{ml})^2\geq A^{-2ml}M_{2ml}\geq A^{-2ml}(M_{2l})^m\geq  A^{-4ml-2m}M_2^{-m}(M_{2l+2})^m$. Thus, we finally see that there is a constant $C>0$ such that
\begin{equation}\label{noreg}
\left\vert\frac{\partial^{2l+2} F}{\partial x^{2l+2}}(0,y_l)\right\vert\geq C^{l+1} (2l+2)!(M_{2l+2})^m,\, \text{ with }\, \lim_{l\to\infty}y_l=0, 
\end{equation}
which clearly implies $F\notin\mathcal{C}_M(\mathbb{R}^2,0)$. The existence of a similar counter-example in one variable is now a direct consequence of the results in \cite[Section 3]{KMR}: starting from \eqref{noreg}, it is possible to construct a curve $\gamma:\mathbb{R}\to \mathbb{R}^2$, with components in $\mathcal{C}_M(\mathbb{R})$, such that $\gamma(0)=0$ and $F\circ\gamma\notin \mathcal{C}_M(\mathbb{R},0)$. Thus, setting $f=F\circ\gamma$, we have $f^p= (F)^p\circ\gamma\in \mathcal{C}_M(\mathbb{R},0)$ and $f\notin\mathcal{C}_M(\mathbb{R},0)$.
\end{proof}

As in the classic $\mathcal{C}^\infty$ case of Joris's theorem, it turns out, however, that a positive result can be obtained with assumptions on two suitable powers of $f$. 

\subsection{Joris's theorem for Denjoy-Carleman classes}
Due to the local nature of the problem, it is convenient to also state the main result of this article in terms of function germs.  
\begin{thm}\label{main} 
Let $M$ be a weight sequence that satisfies the moderate growth condition. Let $f$ be a germ of complex-valued function at the origin in $\mathbb{R}$. Assume there is a couple $(p,q)$ of non-zero natural integers with $\gcd(p,q)=1$ such that both germs $f^p$ and $f^q$ belong to $\mathcal{C}_M(\mathbb{R},0)$. Then $f$ belongs to $\mathcal{C}_M(\mathbb{R},0)$.  
\end{thm}

Postponing the proof to Sections \ref{technical} and \ref{final}, we shall devote the rest of the present section to comments and corollaries.

\begin{rem}
Obviously, the above statement implies that if $\Omega$ is an open subset of $\mathbb{R}$ and $f:\Omega\to \mathbb{C}$ is a function such that $f^p$ and $f^q$ belong to $\mathcal{C}_M(\Omega)$, with $\gcd(p,q)=1$, then $f$ belongs to $\mathcal{C}_M(\Omega)$. 
\end{rem}

\begin{rem} 
The result is no longer true without the moderate growth assumption. A counter-example is once again provided by the functions $g_\lambda$ defined in \eqref{glam}. Indeed, assume for instance $p<q$ and set $f=g_{p\lambda}$. We then have $f^p=g_\lambda\in\mathcal{C}_{M^\lambda}(\mathbb{R},0)$ and $f^q=g_{\lambda'}\in \mathcal{C}_{M^{\lambda'}}(\mathbb{R},0)$ with $\lambda'=\frac{p}{q}\lambda<\lambda$, hence $f^q\in \mathcal{C}_{M^\lambda}(\mathbb{R},0)$. However $f$ does not belong to $\mathcal{C}_{M^\lambda}(\mathbb{R},0)$. 
\end{rem}

\begin{rem} 
As already mentioned in Section \ref{background}, the quasianalytic case does not require moderate growth, but the much weaker assumption of stability under derivation, and the result can then be obtained by straightforward arguments. The interest of Theorem \ref{main} therefore lies in the non-quasianalytic case, although non-quasianalyticity will not be used in the proof. 
\end{rem}

As noticed in the article of Joris \cite{Jor}, in the $\mathcal{C}^\infty$ case, a generalization to functions of several variables is immediate, thanks to the classical result of Boman \cite{Bom} stating that $\mathcal{C}^\infty$ smoothness can be tested along curves. Analogously, for non-quasianalytic classes, the contents of \cite[Section 3]{KMR} immediately yield the following corollary of Theorem \ref{main}. 

\begin{cor} 
Let $M$ be a weight sequence that satisfies the moderate growth and non-quasianalyticity conditions. Let $f$ be a germ of complex-valued function at the origin in $\mathbb{R}^n$. Assume there is a couple $(p,q)$ of non-zero natural integers with $\gcd(p,q)=1$ such that both germs $f^p$ and $f^q$ belong to $\mathcal{C}_M(\mathbb{R}^n,0)$. Then $f$ belongs to $\mathcal{C}_M(\mathbb{R}^n,0)$.  
\end{cor}

The quasianalytic case if of a different nature and the results in \cite{Jaf} and \cite{Rai2} show that it cannot be treated directly by an argument of reduction to lower dimensions. The particular situation of quasianalytic classes obtained as intersections of non-quasianalytic ones as in \cite{KMR2} does not seem more immediately tractable, as the classes defining the intersections may not have suitable properties of logarithmic convexity or moderate growth. \\ 

We now proceed with the proof of Theorem \ref{main}. 

\section{Preparations}\label{technical}
\subsection{Uniform estimates for Cauchy-Riemann equations}\label{dbarsol}
In what follows, for $1\leq p\leq \infty$, we denote by $\Vert\cdot\Vert_p$ the usual norm on the space $L^p(\mathbb{C})$ associated with the standard Lebesgue measure $\lambda$. For $z\in \mathbb{C}$ and $r>0$, we denote by $D(z,r)$ the open disk $\{\zeta\in\mathbb{C}: \vert z-\zeta\vert<r\}$. We write $\mathbbm{1}_A$ for the indicator function of a set $A$. 

Let $\mathcal{K}$ denote the Cauchy kernel in $\mathbb{C}$, that is, $\mathcal{K}(z)=\frac{1}{\pi z}$. 
Let $U$ be a bounded open subset of $\mathbb{C}$. By elementary arguments, for any element $w$ of $L^\infty(\mathbb{C})$ such that $w=0$ in $\mathbb{C}\setminus U$, the convolution $v=\mathcal{K}*w$ defines a bounded continuous function in $\mathbb{C}$ that satisfies $\partial v/\partial\bar{z}=w$
 in the sense of distributions in $\mathbb{C}$, and
\begin{equation}\label{estimconvol1}
\Vert v\Vert_\infty\leq C \Vert w\Vert_\infty
\end{equation}
for some suitable constant $C$ depending only on $\max_{\zeta\in U}\vert \zeta\vert$. In order to follow the pattern of \cite{Fed}, more subtle uniform estimates on $v$ are needed. These estimates are described by the following lemma.    

\begin{lem}\label{estimconvol2}
Let $U$, $w$ and $v$ be as above. Then for any real number $r\in (0,\frac{1}{2}]$ and any $z\in U$, we have  
\begin{equation*}
\vert v(z)\vert\leq C \left(r \Vert w\Vert_\infty+\left(\vert\ln r\vert\right)^{1/2}\Vert w\Vert_2\right)
\end{equation*}
for some suitable constant $C$ depending only on $\max_{\zeta\in U}\vert \zeta\vert$.
\end{lem}
\begin{proof} For the reader's convenience, we include the proof sketched in \cite{Fed}. Choose $R\geq 1$  such that $U\subset D\big(0,\frac{R}{2}\big)$. For $z\in U$ and $\vert \zeta\vert \geq R$ we have $\vert z-\zeta\vert >\frac{R}{2}$, hence $w(z-\zeta)=0$. We can therefore write
$v(z)=\int_{D(0,R)}\mathcal{K}(\zeta)w(z-\zeta)\ud\lambda(\zeta)= \int_{D(0,r)}\mathcal{K}(\zeta)w(z-\zeta)\ud\lambda(\zeta)+\int_{\{r\leq\vert \zeta\vert<R\}}\mathcal{K}(\zeta)w(z-\zeta)\ud\lambda(\zeta)$. A crude majorization immediately yields $\left\vert \int_{D(0,r)}\mathcal{K}(\zeta)w(z-\zeta)\ud\lambda(\zeta)\right\vert\leq \int_{D(0,r)}\frac{\ud\lambda(\zeta)}{\pi\vert \zeta\vert}\Vert w\Vert_\infty=2r\Vert w\Vert_\infty$. By the Cauchy-Schwarz inequality, we also have $\left\vert \int_{\{r\leq\vert \zeta\vert< R\}}\mathcal{K}(\zeta)w(z-\zeta)\ud\lambda(\zeta)\right\vert\leq \left(\int_{\{r\leq\vert \zeta\vert< R\}}\frac{\ud\lambda(\zeta)}{\pi^2\vert \zeta\vert^2}\right)^{1/2}\Vert w\Vert_2= \big(\frac{2}{\pi}\ln({R}/{r})\big)^{1/2}\Vert w\Vert_2$. The result easily follows. 
\end{proof}

\subsection{Technical estimates in ellipses}\label{objects}  
\begin{defin}\label{ellipses}
 For any $\ve>0$, we put $\Omega_\ve=\varphi_\ve (S)$, where $S$ is the strip $\{z\in\mathbb{C} :\vert \Im z\vert<1\}$ and $\varphi_\ve$ is the mapping of the complex plane defined by $\varphi_\ve(z)=\sin (\ve z)$. 
\end{defin}
In other words, the open set $\Omega_\ve$ is the interior of the ellipse with vertices $\pm \cosh\ve$ and co-vertices $ \pm i\sinh\ve $. It contains the real interval $[-1,1]=\varphi_\ve(\mathbb{R})$. and becomes narrower as $\ve$ tends to $0$. \\

The following covering lemma is elementary.    

\begin{lem}\label{cover}
For any real number $\ve$ with $0<\ve\leq 1$, there is a radius $\eta_\ve>0$ and a finite family of disks $D(z_{j,\ve},\eta_\ve)$, $j=1,\ldots,N_\ve$, with the following properties:
\begin{equation}\label{cover1}
\Omega_{\ve/2}\subset\bigcup_{j=1}^{N_\ve}D(z_{j,\ve}, \eta_\ve),
\end{equation}
\begin{equation}\label{cover2}
\overline{D(z_{j,\ve},2\eta_\ve)}\subset \Omega_\ve\, \text{ for }\, j=1,\ldots,N_\ve,
\end{equation}
\begin{equation}\label{cover3}
N_\varepsilon \leq C \ve^{-3}\, \text{ for some absolute constant }C. 
\end{equation}
\end{lem}
\begin{proof} Basic arguments show that $\dist(\partial\Omega_{\ve/2}, \partial\Omega_{\ve})\geq \frac{1}{4} \ve^2$. Thus, any closed disk of radius $\frac{1}{8}\ve^2$ that intersects $\Omega_{\ve/2}$ is contained in $\Omega_\ve$. Set $\eta_\ve=\frac{1}{16}\ve^2$ and notice that $\Omega_{\ve/2}$ is contained in a rectangle of length $ 2\cosh(\ve/2)$ and width $2\sinh(\ve/2)$. It is an easy exercise to check that such a rectangle can be covered by a family $\mathcal{F}_\ve$ of open disks of radius $\eta_\ve$ with $\card{\mathcal{F}_\ve} \leq C\varepsilon^{-3}$ for some absolute constant $C$. Keeping only the elements of $\mathcal{F}_\ve$ that intersect $\Omega_{\ve/2}$, we obtain a family of disks having all the desired properties. 
\end{proof}

We can now obtain technical estimates following closely a key statement in \cite{Fed}, with slight modifications required in our framework. For the reader's convenience, we give a complete proof. 

\begin{lem}\label{l2estim} 
Let $\ve$ be a real number with $0<\ve\leq 1$, let $g$ be a bounded holomorphic function in $\Omega_\ve$, and let $K$ be a real number such that $\vert g\vert\leq K$ in $\Omega_\ve$. For any real number $r>0$, we have 
\begin{equation*}
\int_{\Omega_{\ve/2}}\vert g'\vert^2 \mathbbm{1}_{\{\vert g\vert<r\}} \ud\lambda \leq C\frac{r^2}{\ve^3}\ln\left(\frac{K^2}{r^2}+1\right)
\end{equation*}
for some absolute constant $C$. 
\end{lem}
\begin{proof} 
For $j=1,\ldots, N_\ve$, consider the disk $D(z_{j,\ve},\eta_\ve)$ of Lemma \ref{cover}. It is easy to see that 
\begin{equation}\label{chvar}
\int_{D(z_{j,\ve},\eta_\ve)} \vert g'\vert^2 \mathbbm{1}_{\{\vert g\vert<r\}} \ud\lambda=\int_{D(0,\frac{1}{2})}\vert g_{j,\ve}'\vert^2 \mathbbm{1}_{\{\vert g_{j,\ve}\vert<r\}} \ud\lambda
\end{equation}
where $g_{j,\ve}$ is defined by
\begin{equation*}
g_{j,\ve}(\zeta)=g(z_{j,\ve}+2\eta_\ve \zeta).
\end{equation*}
Property \eqref{cover2} and the assumptions on $g$ ensure that the function $g_{j,\ve}$ is holomorphic in a neighborhood of $\overline{D(0,1)}$. Set
\begin{equation*}
\Psi_{j,\ve}=\ln \left(\vert g_{j,\ve}\vert^2+r^2\right). 
\end{equation*}
Then $\Psi_{j,\ve}$ is a smooth subharmonic function in a neighborhood of $\overline{D(0,1)}$ and its Laplacian is  
\begin{equation*}
\Delta \Psi_{j,\ve}=4r^2 \frac{\vert g_{j,\ve}'\vert^2}{(\vert g_{j,\ve}\vert^2+r^2)^2}.
\end{equation*}
In particular, we have $\Delta \Psi_{j,\ve}\geq \frac{1}{r^2}\vert g_{j,\ve}'\vert^2 \mathbbm{1}_{\{\vert g_{j,\ve}\vert<r\}}$. Thus, we get
\begin{equation}\label{majorint}
\begin{split}
\int_{D(0,\frac{1}{2})}\vert g_{j,\ve}'\vert^2 \mathbbm{1}_{\{\vert g_{j,\ve}\vert<r\}} \ud\lambda &\leq  r^2\int_{D(0,\frac{1}{2})} \Delta \Psi_{j,\ve} \ud\lambda \\
& \leq \frac{r^2}{\ln 2} \int_{D(0,\frac{1}{2})} \Delta \Psi_{j,\ve}(\zeta)\ln\left(\frac{1}{\vert \zeta\vert}\right)\ud\lambda(\zeta)\\
& \leq \frac{r^2}{\ln 2} \int_{D(0,1)} \Delta \Psi_{j,\ve}(\zeta)\ln\left(\frac{1}{\vert \zeta\vert}\right)\ud\lambda(\zeta).
\end{split}
\end{equation}
Using Green's formula for the Laplacian, together with the obvious estimates $\Psi_{j,\ve}\leq \ln (K^2+r^2)$ and $\Psi_{j,\ve}(0)\geq \ln r^2$, we see that 
\begin{equation}\label{green}
\begin{split}
\int_{D(0,1)}\Delta \Psi_{j,\ve}(\zeta)\ln\left(\frac{1}{\vert \zeta\vert}\right)\ud\lambda(\zeta) &=\int_0^{2\pi} \Psi_{j,\ve}(e^{i\theta})\ud\theta- 2\pi \Psi_{j,\ve}(0) \\
&\leq 2\pi (\ln(K^2+r^2)-\ln r^2).
\end{split}
\end{equation}
Gathering \eqref{chvar}, \eqref{majorint} and \eqref{green}, we obtain 
\begin{equation*}
\int_{D(z_{j,\ve},\eta_\ve)} \vert g'\vert^2 \mathbbm{1}_{\{\vert g\vert<r\}} \ud\lambda\leq \frac{2\pi}{\ln 2} \ln\left(\frac{K^2}{r^2}+1\right).
\end{equation*}
Together with \eqref{cover1} and \eqref{cover3}, this implies the desired result. 
\end{proof}

We end this section with a lemma which, roughly speaking, means that for bounded holomorphic functions in $\Omega_\ve$, a suitable property of ``smallness'' on the interval $[-1,1]$ still holds in $\Omega_{\ve/2}$, up to constants.  

\begin{lem}\label{threelines}
Let $\ve$ be a positive real number and let $g$ be a function holomorphic in $\Omega_\ve$ and continuous up to the boundary. Assume that the weight sequence $M$ satisfies the moderate growth property \eqref{modg}, and let $L$, $a_1$ and $a_2$ be positive numbers such that  
\begin{equation*}
\vert g\vert \leq L \text{ in }\, \Omega_\ve\quad \text{and }\quad\vert g\vert \leq a_1 h_M(a_2\ve)\, \text{ on }\, [-1,1].
\end{equation*}
Then we have 
\begin{equation*}
\vert g\vert\leq  a_3 h_M(a_4 \ve)\, \text{ in }\, \Omega_{\ve/2},
\end{equation*} 
for suitable positive numbers $a_3$ and $a_4$ depending only on $L$, $a_1$, $a_2$ and on the sequence $M$.
\end{lem}

\begin{proof} With the notation of Definition \ref{ellipses}, put $f=\frac{1}{a_1}g\circ \varphi_\ve$. The function $f$ is holomorphic in the strip $S$ and continuous up to the boundary. Setting $K=\max(1,\frac{L}{a_1})$, we have $\vert f\vert \leq K$ in $S$ and $\vert f\vert\leq h_M(a_2\ve)$ on $\mathbb{R}$. Using Hadamard's three-lines theorem \cite[pp. 33--34]{RS}, we get $\vert f(z)\vert\leq (h_M(a_2\ve))^{1-\vert \Im z\vert}K^{\vert \Im z\vert}$ for every $z\in S$. Notice that $h_M(a_2\ve)\leq 1$ and $K\geq 1$. Since any point $w$ in $\Omega_{\ve/2}$ can be written $w=\varphi_\ve(z)$ with $z\in S$ and $\vert\Im z\vert\leq 1/2$, we therefore get the estimate $\vert g(w)\vert\leq a_1(K h_M(a_2\ve))^{1/2}$ for any such $w$. Since $M$ has moderate growth, it then suffices to use \eqref{hfunct2} to obtain the desired result, with $a_3=\max(a_1^{1/2},L^{1/2})$ and $a_4=\kappa_2a_2$. 
\end{proof}

\subsection{An approximation-theoretic characterization of ultradifferentiable functions}
The approach of Joris's theorem in \cite{Fed} relies on a characterization of 
$\mathcal{C}^k$ regularity of a  function $f$ on a bounded interval $I$ in terms of the rate of approximation of $f$ by uniformly bounded families of holomorphic functions in 
narrow neighborhoods of $I$ in $\mathbb{C}$. In this section, we obtain, in the same spirit, a characterization of $\mathcal{C}_M$ regularity under the moderate growth assumption.  

\begin{defin}\label{pm}
Let $M$ be a weight sequence. We shall say that a complex-valued function $f$ defined on $[-1,1]$ satisfies property $(\mathcal{P}_M)$ if there are positive constants $K$, $c_1$, $c_2$ and a family $(f_\ve)_{0<\ve\leq \ve_0}$ of continuous functions 
in $\mathbb{C}$ such that, for any $\ve\in (0,\ve_0]$, the following conditions are satisfied:
\begin{align}
& \text{the function } f_\ve \text{ is holomorphic in } \Omega_\ve, \label{pm1}\\
& \vert f_\ve\vert \leq K\, \text{ in }\, \Omega_\ve, \label{pm2} \\
& \vert f-f_\ve\vert\leq c_1 h_M(c_2\ve)\, \textrm{ on }\, [-1,1]. \label{pm3}
\end{align}
\end{defin}

\begin{prop}\label{approx}
Every element of $\mathcal{C}_M([-1,1])$ satisfies property $(\mathcal{P}_M)$. Conversely, if a complex-valued function defined on $[-1,1]$ satisfies $(\mathcal{P}_M)$, then it belongs to $\mathcal{C}_M([-b,b])$ for any real number $b$ with $0<b<1$. 
\end{prop}
\begin{proof} Let $f$ be an element of $\mathcal{C}_M([-1,1])$. By Dynkin's theorem on $\bar\partial$-flat extensions \cite{Dy}, there are positive constants $c_1$ and $c_2$, and a function $g$ of class $\mathcal{C}^1$ with compact support in $\mathbb{C}$, such that $g=f$ on $[-1,1]$ and, for any $z\in \mathbb{C}$,
\begin{equation}\label{dbarflat}
\left\vert\frac{\partial g}{\partial \bar z}(z)\right\vert\leq c_1 h_M(c_2 \dist(z,[-1,1])).
\end{equation}
For every $\ve\in (0,1]$, put 
\begin{equation*} 
w_\ve=\mathbbm{1}_{\Omega_\ve}\frac{\partial g}{\partial \bar z}.
\end{equation*}
Then $w_\ve$ is an element of $L^\infty(\mathbb{C})$, with $ w_\ve=0$ in $\mathbb{C}\setminus \Omega_\ve$. Besides, it is easy to see that for  $z\in \Omega_\ve$, we have $\dist(z,[-1,1]) \leq C \ve$ for some absolute constant $C$. After multiplying $c_2$ by $C$, \eqref{dbarflat} implies
\begin{equation}\label{majin2e}
\Vert w_\ve\Vert_\infty\leq c_1 h_M(c_2\ve).
\end{equation}
Now, set $v_\ve=\mathcal{K}*w_\ve$ where $\mathcal{K}$ is the Cauchy kernel. As explained in Section \ref{dbarsol}, $v_\ve$ is a continuous function in $\mathbb{C}$ such that
$\partial v_\ve/\partial\bar{z}=w_\ve$
in the sense of distributions in $\mathbb{C}$, hence
\begin{equation}\label{eqd1}
\frac{\partial v_\ve}{\partial \bar z}=\frac{\partial g}{\partial \bar z}\, \text{ in }\, \Omega_\ve.
\end{equation}
Moreover, by \eqref{estimconvol1} and \eqref{majin2e}, it satisfies
\begin{equation}\label{majsol}
\Vert v_\ve\Vert_\infty\leq c_1 h_M(c_2\ve)
\end{equation}
after multiplying $c_1$ by a suitable absolute constant. Define $f_\ve=g-v_\ve$. Then $f_\ve$ is a bounded continuous function in $\mathbb{C}$ and we have $\Vert f_\ve\Vert_\infty\leq \Vert g\Vert_\infty + c_1 h_M(c_2\ve)$, hence \eqref{pm2} with $K=  \Vert g\Vert_\infty + c_1 h_M(c_2)$. By \eqref{eqd1}, we have $ \partial f_\ve/\partial\bar{z}= 0$ in $\Omega_\ve$, hence \eqref{pm1}. Finally,  \eqref{majsol} implies \eqref{pm3} since $f$ and $g$ coincide on $[-1,1]$. Thus, property $(\mathcal{P}_M)$ is established, with $\ve_0=1$.    

Conversely, let $f:[-1,1]\to \mathbb{C}$ be a function that satisfies $(\mathcal{P}_M)$. For $0<\ve\leq \ve_0/2$, it is readily seen that the function $f_\ve-f_{2\ve}$ meets the assumptions of Lemma \ref{threelines} with $L=2K$, $a_1=2c_1$ and $a_2=2c_2$. We therefore get 
\begin{equation}\label{majg}
\vert f_\ve-f_{2\ve}\vert\leq  a_3 h_M(a_4 \ve)\, \text{ in }\, \Omega_{\ve/2},
\end{equation}
for some suitable constants $a_3$ and $a_4$ depending only on $K$, $c_1$ and $c_2$
Now, let $b$ be a real number with $0<b<1$. By elementary geometric considerations, there is an absolute positive constant $C$ such that for any $x\in [-b,b]$, the  closed disk centered at $x$ with radius $C(b-1)\ve$ is contained in $\Omega_{\ve/2}$. Using the Cauchy formula and \eqref{majg}, we therefore get $\vert (f_\ve-f_{2\ve})^{(j)}(x)\vert \leq a_3(C(b-1))^{-j} j! \ve^{-j} h_M(a_4\ve)$ for any $x\in [-b,b]$ and any $j\in \mathbb{N}$. Taking \eqref{hfunct3} into account, we get
\begin{equation}
\Vert f_\ve-f_{2\ve}\Vert_{[-b,b],\sigma}\leq a_3h_M(a_5\ve)
\end{equation}
with $\sigma=\kappa_2a_4(C(b-1))^{-1}$ and $a_5=\kappa_2a_4$. Since $h_M(a_5\ve)\leq a_5M_1\ve$, this clearly implies the absolute convergence of the series $f_{\ve_0}+\sum_{j\geq 1} \big(f_{\ve_02^{-j}}-f_{\ve_02^{-(j-1)}}\big)$ in the Banach space $\mathcal{C}_{M,\sigma}([-b,b])$. Let $g$ denote its sum. For every integer $J\geq 1$, we have 
\begin{equation*}
g=f_{\ve_02^{-J}}+\sum_{j\geq J+1} \big(f_{\ve_02^{-j}}-f_{\ve_02^{-(j-1)}}\big).
\end{equation*}
For $x\in [-b,b]$, we infer $\vert f(x)-g(x)\vert\leq \big\vert f(x)-f_{\ve_02^{-J}}(x)\big\vert+\sum_{j\geq J+1} \big\vert f_{\ve_02^{-j}}(x)-f_{\ve_02^{-(j-1)}}(x)\big\vert \leq c_1h_M(c_2\ve_02^{-J})+\sum_{j\geq J+1} \big\Vert f_{\ve_02^{-j}}-f_{\ve_02^{-(j-1)}}\big\Vert_{[-b,b],\sigma} $. Letting $J$ tend to $\infty$, we obtain $f(x)=g(x)$, hence $f\in\mathcal{C}_M([-b,b])$. 
\end{proof}
\begin{rem} 
The moderate growth assumption is crucial in the proof of the converse part of Proposition \ref{approx}, but the fact that the elements of $\mathcal{C}_M([-1,1])$ satisfy property $(\mathcal{P}_M)$ is still true under the weaker condition \eqref{stabder} of stability under derivation, which is required by Dynkin's result on $\bar\partial$-flat extensions.   
\end{rem}

\section{Proof of the main result}\label{final}
\subsection{Reduction to a special case} Consider two positive integers $p$ and $q$ such that $\gcd(p,q)=1$ and let $f$ be a function germ at the origin in $\mathbb{R}$ such that $f^p$ and $f^q$ belong to $\mathcal{C}_M(\mathbb{R},0)$. Up to a linear change of variable, we can assume that $f^p$ and $f^q$ belong to $\mathcal{C}_M([-1,1])$. One can easily find $m\in\mathbb{N}$ such that any integer $j\geq m$ can be written  
$j=pk+ql$ with $(k,l)\in\mathbb{N}^2$. We then have $f^j=(f^p)^k(f^q)^l$ and, since $\mathcal{C}_M([-1,1])$ is an algebra, we see that $f^j$ belongs to $\mathcal{C}_M([-1,1])$. In particular, we have 
\begin{equation}\label{red}
f^m\in \mathcal{C}_M([-1,1])\, \text{ and }\, f^{m+1} \in\mathcal{C}_M([-1,1]).
\end{equation}
In order to conclude that $f$ belongs to $\mathcal{C}_M(\mathbb{R},0)$, it then suffices to prove that \eqref{red} implies $f\in \mathcal{C}_M([-b,b])$ for $0<b< 1$.  

\subsection{Construction of approximants} 
By Proposition \ref{approx}, there are constants $K\geq 1$, $c_1>0$, $c_2>0$ and families $(g_\ve)_{0<\ve\leq \ve_0}$ and $(h_\ve)_{0<\ve\leq \ve_0}$ of bounded continuous functions in $\mathbb{C}$ such that for $0<\ve\leq \ve_0$, we have the following properties:
\begin{align} 
& \text{the functions } g_\ve\text{ and } h_\ve \text{ are holomorphic in }\Omega_{\ve},  \\ 
& \vert g_\ve\vert \leq K\, \text{ and }\, \vert h_\ve\vert_\infty\leq K\, \text{ in }\, \Omega_{\ve}, \label{fgh2}  \\ 
& \vert f^m-g_\ve\vert\leq c_1h_M(c_2\ve)\, \text{ and }\, \vert f^{m+1}-h_\ve\vert\leq c_1h_M(c_2\ve)\,\text{ on }\, [-1,1]. \label{fgh3}
\end{align}
In view of the above, the intuitive candidate for an holomorphic approximation of $f$ on $[-1,1]$ is the quotient ${h_\ve}/{g_\ve}$, but it has to be modified to avoid small denominators. We therefore define \begin{equation*}
u_\ve=\chi_\ve \frac{\overline{g_\ve}h_\ve}{(\max(\vert g_\ve\vert, r_\ve))^2}
\end{equation*}
where $r_\ve$ is a positive real number, and $\chi_\ve:\mathbb{C}\to [0,1]$ is a smooth cutoff function with $\chi_\ve=1$ in $\Omega_{\ve/2}$ and $\supp\chi_\ve\subset \Omega_{\ve}$. The function $u_\ve$ is well-defined, continuous with compact support in $\mathbb{C}$ and it coincides with $ h_\ve/g_\ve$ in $\Omega_{\ve/2} \cap \{\vert g_\ve\vert>r_\ve\}$, but it is obviously not holomorphic in a whole neighborhood of $[-1,1]$. In the rest of the proof, we shall however see that for a suitable choice of $r_\ve$, this function satisfies uniform bounds and is ``close enough'' to $f$ on $[-1,1]$, and we shall then recover a holomorphic approximant \emph{via} a $\bar\partial$-problem.\\ 

Using  \eqref{fgh2}, \eqref{fgh3} and the elementary inequality 
$\vert z^j-\zeta^j\vert\leq j\max(\vert z\vert, \vert \zeta\vert)^{j-1}\vert z-\zeta\vert$  
with $j=m$ and with $j=m+1$, we see that there is a constant $c_3$ depending only on $K$, $c_1$ and $m$,  such that $\vert h_\ve^m-g_\ve^{m+1}\vert \leq c_3h_M(c_2\ve)$ on $[-1,1]$. Moreover, $h_\ve^m-g_\ve^{m+1}$ is holomorphic in $\Omega_\ve$, continuous up to the boundary and we have $\vert h_\ve^m-u_\ve^{m+1}\vert\leq 2K^{m+1}$ in $\Omega_\ve$. Thus, applying Lemma \ref{threelines} with $L=2K^{m+1}$, $a_1=c_3$ and $a_2=c_2$, we obtain 
\begin{equation}\label{hg}
\vert h_\ve^m-g_\ve^{m+1}\vert \leq c_4h_M(c_5\ve)\, \text{ in }\, \Omega_{\ve/2},
\end{equation}
where $c_4$ and $c_5$ depend only on $K$, $c_1$, $c_2$ and $m$. We shall now set
\begin{equation}\label{defdelta}
\delta_\ve=c_4 h_M(c_5\ve)\ \text{ and }\ r_\ve = \delta_\ve^\frac{1}{m+1}.
\end{equation}
Since we can obviously assume 
$c_4\geq c_1$ and $c_5\geq c_2$, it is convenient to rewrite \eqref{fgh3} and \eqref{hg} as 
\begin{equation}\label{delta}
\begin{split}
&\vert f^{m+1}-h_\ve\vert\leq \delta_\ve\, \text{ and }\, \vert f^m-g_\ve\vert\leq \delta_\ve\, \text{ on }\, [-1,1],  \\
&\vert h_\ve^m-g_\ve^{m+1}\vert\leq \delta_\ve\, \text{ in }\,  \Omega_{\ve/2}. 
\end{split}
\end{equation}
Also, notice that we have $\delta_\ve\leq r_\ve\leq 1$ for $\ve$ small enough.

\begin{lem}\label{ubound}
For any sufficiently small $\ve>0$,  we have  
\begin{equation*}
\vert u_\ve\vert \leq (2K)^{1/m}\, \text{ in }\, \Omega_{\ve/2}.
\end{equation*}
\end{lem}
\begin{proof} By \eqref{delta}, in $\Omega_{\ve/2}$, we have $\vert h_\ve\vert\leq \vert(\vert g_\ve^{m+1}\vert+\vert h_\ve^m-g_\ve^{m+1}\vert)^{1/m} \leq  (\vert g_\ve\vert^{m+1}+r_\ve^{m+1})^{1/m}\leq 2^{1/m} (\max(\vert g_\ve\vert, r_\ve))^\frac{m+1}{m}$, hence $\vert u_\ve\vert\leq 2^{1/m} \vert g_\ve\vert (\max(\vert g_\ve\vert, r_\ve))^{-1+\frac{1}{m}}\leq 2^{1/m}(\max(\vert g_\ve\vert, r_\ve))^\frac{1}{m}$. The result then follows from \eqref{fgh2}.  
\end{proof}

\begin{lem}\label{fmu} 
There is a constant $c_6$ depending only on $K$ and $m$, such that, for any sufficiently small $\ve>0$,  we have  
\begin{equation*} 
\vert f -u_\ve\vert\leq c_6 \delta_\ve^\frac{1}{m(m+1)}\, \text{ on }\, [-1,1].
\end{equation*}
\end{lem} 
\begin{proof} The estimate will be proved separately on the sets $F_\ve=[-1,1]\cap\{\vert g_\ve\vert\leq r_\ve\}$ and $G_\ve=[-1,1]\cap\{\vert g_\ve\vert>r_\ve\}$. On the set $F_\ve$, we have
$f-u_\ve= f- r_\ve^{-2}\, \overline{g_\ve}h_\ve$, hence 
\begin{equation*}
\vert f-u_\ve\vert\leq \vert f\vert+r_\ve^{-2}\,\vert g_\ve\vert \vert h_\ve\vert\leq  \vert f\vert+r_\ve^{-1} \vert h_\ve\vert.
\end{equation*}
By \eqref{delta}, we also have
$\vert f\vert\leq (\vert g_\ve\vert+\vert f^m -g_\ve\vert)^{1/m}\leq (r_\ve+\delta_\ve)^{1/m}\leq (2r_\ve)^{1/m}$ and $\vert h_\ve\vert\leq (\vert g_\ve^{m+1}\vert+\vert h_\ve^m-g_\ve^{m+1}\vert)^{1/m}\leq (r_\ve^{m+1}+\delta_\ve)^{1/m}=(2r_\ve^{m+1})^{1/m}=r_\ve (2r_\ve)^{1/m}$. Setting $c_7=2^{1+\frac{1}{m}}$, we finally derive
\begin{equation}\label{surF}
\vert f-u_\ve\vert\leq c_7 r_\ve^{1/m}= c_7\delta_\ve^\frac{1}{m(m+1)}\, \text{ on }\, F_\ve.
\end{equation}
On the set $G_\ve$, we have
\begin{equation*}
 f-u_\ve = f-\frac{h_\ve}{g_\ve}= \frac{f(g_\ve-f^m)+f^{m+1}-h_\ve}{g_\ve}
\end{equation*}
with
$\vert f\vert\leq (\vert g_\ve\vert+\vert f^m -g_\ve\vert)^{1/m}\leq (K+\delta_\ve)^{1/m}\leq (K+1)^{1/m}$. Thus, using \eqref{delta}, it is easy to obtain
\begin{equation}\label{surG}
\vert f-u_\ve\vert\leq c_8 \frac{\delta_\ve}{r_\ve}=c_8\delta_\ve^\frac{m}{m+1}\, \text{ on }\, G_\ve,
\end{equation}
with $c_8=(K+1)^{1/m}+1$. The lemma clearly follows from \eqref{surF} and \eqref{surG}. 
\end{proof}

Now we proceed to obtain a holomorphic modification of $u_\ve$. As a starting point, we need basic information on $\partial u_\ve/\partial\bar{z}$. 

\begin{lem} The distributional derivative $\partial u_\ve/\partial\bar{z}$ is an element of $L^\infty(\mathbb{C})$ and we have 
\begin{equation}\label{calcdbar} 
\frac{\partial u_\ve}{\partial\bar{z}}=\frac{1}{r_\ve^2}\,\overline{g_\ve'}h_\ve \mathbbm{1}_{\{\vert g_\ve\vert<r_\ve\}}\, \text{ in }\, \Omega_{\ve/2}.
\end{equation}
\end{lem}
\begin{proof} We introduce the sets $X_\ve=\Omega_{\ve/2}\cap\{\vert g_\ve\vert<r_\ve\}$, $Y_\ve=\Omega_{\ve/2}\cap\{\vert g_\ve\vert>r_\ve\}$ and $Z_\ve=\Omega_{\ve/2}\cap\{\vert g_\ve\vert=r_\ve\}$. Since $g_\ve$ is holomorphic in $\Omega_\ve$, either the set $Z_\ve$ has measure zero, or $g_\ve$ is constant. In the latter case, $u_\ve$ is a constant times $h_\ve$ and the conclusion of the lemma is immediate. We therefore focus on the general case of a non-constant $g_\ve$. Since $\supp \chi_\ve\subset \Omega_{\ve}$ and $\vert g_\ve\vert^2$ is smooth in $\Omega_{\ve}$, it is readily seen that the denominator $\max(\vert g_\ve\vert^2, r_\ve^2)$ is Lipschitz and bounded away from zero in a neighborhood of $\supp\chi_\ve$. Taking into account the smoothness of $\overline{g_\ve}h_\ve$ in $\Omega_{\ve}$, we infer that $u_\ve$ is a bounded Lipschitz function in $\mathbb{C}$, hence it belongs to the Sobolev space $W^{1,\infty}(\mathbb{C})$ (see \cite[Proposition 9.3]{Bre} or \cite[Theorem 6.12]{Hei}). Thus, the distribution $\partial u_\ve/\partial\bar{z}$ is an element of $L^\infty(\mathbb{C})$. Since $\Omega_{\ve/2}=X_\ve\cup Y_\ve\cup Z_\ve$ and $Z_\ve$ has measure zero, it then suffices to check \eqref{calcdbar} in each of the open sets $X_\ve$ and $Y_\ve$, which boils down to an explicit computation using the holomorphicity of $g_\ve$ and $h_\ve$ in those sets. In $X_\ve$, we have $u_\ve=r_\ve^{-2}\,\overline{g_\ve}h_\ve$, hence $\partial u_\ve/\partial\bar{z}=r_\ve^{-2}\,\overline{g_\ve'}h_\ve$. In $Y_\ve$, we have $u_\ve=h_\ve/g_\ve$, hence $\partial u_\ve/\partial\bar{z}=0$. The lemma is proved. 
\end{proof}

We now set
\begin{equation*}
w_\ve= \mathbbm{1}_{\Omega_{\ve/2}}\frac{\partial u_\ve}{\partial\bar{z}}\quad \text{and}\quad v_\ve=\mathcal{K}*w_\ve. 
\end{equation*}
The function $w_\ve$ is an element of $L^\infty(\mathbb{C})$ with $w=0$ in $\mathbb{C}\setminus\Omega_{\ve/2}$. Thus, as explained in Section \ref{dbarsol}, $v_\ve$ is a bounded continuous function in $\mathbb{C}$ that satisfies $\partial v_\ve/\partial\bar{z}=w_\ve$ in the sense of distributions in $\mathbb{C}$, hence
\begin{equation}\label{eqd}
\frac{\partial v_\ve}{\partial\bar{z}}=\frac{\partial u_\ve}{\partial\bar{z}}\, \text{ in }\, \Omega_{\ve/2}.
\end{equation} 
The last ingredient of the proof will be an estimate for $v_\ve$ in $\Omega_{\ve/2}$. 

\begin{lem}\label{vbound} 
Let $s$ be a real number, with $s>m(m+1)$. For $\ve>0$ small enough, we have 
\begin{equation*}
\vert v_\ve\vert\leq c_9 \delta_\ve^{1/s}\, \text{ in }\, \Omega_{\ve/2},
\end{equation*}
where $c_9$ is a constant depending only on $K$, $m$ and $s$. 
\end{lem}
\begin{proof} By Lemma \ref{estimconvol2}, there is a constant $C$ such that for any $\ve>0$ small enough, we have
\begin{equation}\label{estimv}
\vert v_\ve\vert\leq C \left(r_\ve \Vert w_\ve\Vert_\infty+\left(\vert\ln r_\ve\vert\right)^{1/2}\Vert w_\ve\Vert_2\right)\, \text{ in }\, \Omega_{\ve/2}.
\end{equation}
Using \eqref{delta}, we see that in the open set $\Omega_{\ve/2}\cap{\{\vert g_\ve\vert<r_\ve\}}$, we have $\vert h_\ve\vert\leq (\vert g_\ve\vert^{m+1}+\delta_\ve)^{1/m}\leq (r_\ve^{m+1}+\delta_\ve)^{1/m}=2^{1/m}r_\ve^{\frac{m+1}{m}}$. This implies
\begin{equation}\label{majorw}
\vert w_\ve\vert \leq 2^{1/m} r_\ve^{\frac{1}{m}-1}\vert g'_\ve\vert\mathbbm{1}_{\{\vert g_\ve\vert<r_\ve\}}.
\end{equation} 
Now recall that $g_\ve$ is holomorphic in $\Omega_\ve$, with $\vert g_\ve\vert\leq K$. 
Since any closed disk of radius $\frac{1}{8}\ve^2$ centered in $\Omega_{\ve/2}$ is contained in $\Omega_\ve$, the Cauchy formula then yields
$\vert g'_\ve\vert \leq 8K \ve^{-2}$ in $\Omega_{\ve/2}$. Together with \eqref{majorw}, this implies the uniform estimate
\begin{equation}\label{winfin}
\Vert w_\ve\Vert_\infty\leq c_{10}\frac{r_\ve^{\frac{1}{m}-1}}{\ve^2},
\end{equation} 
with $c_{10}= 8\cdot 2^{1/m}K$. Using Lemma \ref{l2estim} and \eqref{majorw}, we also get the $L^2$ estimate
\begin{equation}\label{wl2}
\Vert w_\ve\Vert_2\leq c_{11}\frac{r_\ve^{1/m}}{\ve^{3/2}}\left(\ln\left(\frac{K^2}{r_\ve^2}+1\right)\right)^{1/2}
\end{equation} 
for a positive constant $c_{11}$ depending only on $m$. Since $r_\ve=\delta_\ve^\frac{1}{m+1}$ and $\delta_\ve=o(\ve^j)$ for every integer $j\geq 1$, the desired result follows from \eqref{estimv}, \eqref{winfin} and \eqref{wl2}. 
\end{proof}

It is now possible to complete the proof of Theorem \ref{main}. 

\subsection{End of the proof.} We consider $f_\ve=u_{2\ve}-v_{2\ve}$ for $\ve>0$ small enough. 
The function $f_\ve$ is continuous in $\mathbb{C}$, and it is holomorphic in $\Omega_\ve$, since, by \eqref{eqd}, we also have $ \partial f_\ve/\partial\bar{z}=0$ in the sense of distributions in $\Omega_\ve$. Lemma \ref{ubound} and Lemma \ref{vbound} imply
\begin{equation*}
\vert f_\ve \vert\leq K'\, \text{ in }\, \Omega_\ve,
\end{equation*}
with $K'=(2K)^{1/m}+c_9$. Finally, choose a real number $s$ with $s>m(m+1)$. By Lemma \ref{fmu} and Lemma \ref{vbound}, we have $\vert f- f_\ve\vert \leq \vert f-u_\ve\vert+\vert v_\ve\vert\leq c_{12}\delta_{2\ve}^{1/s}$ on $[-1,1]$, for some suitable constant $c_{12}>0$. Using \eqref{defdelta} and the moderate growth property \eqref{hfunct2}, we get $\delta_{2\ve}^{1/s}\leq c_{13} h_M(c_{14}\ve)$ with $c_{13}=c_4^{1/s}$ and $c_{14}=2\kappa_sc_5$. Thus, we obtain 
\begin{equation*}
\vert f-f_\ve\vert\leq c'_1 h_M(c'_2\ve)\, \text{ on }\, [-1,1],
\end{equation*} with $c'_1=c_{12}c_{13}$ and $c'_2= c_{14}$. We have therefore proved that, for $\ve'_0$ small enough, the family $(f_\ve)_{0<\ve\leq \ve'_0}$ meets the requirements of property $(\mathcal{P}_M)$. Thus, by Proposition \ref{approx}, the function $f$ belongs to $\mathcal{C}_M([-b,b])$ for any $b$ with $0<b<1$, and Theorem \ref{main} is now established.

\end{document}